\documentclass[reqno,12pt]{amsart}
\usepackage{graphics}
\usepackage{amssymb,amsmath}
\input cyracc.def
\font\tencyr=wncyr8
\def\cyr{\tencyr\cyracc}
\newsavebox{\ukrB}
\savebox{\ukrB}{{\cyr B}}
\newcommand{\accessed}{9 February 2009}
\newcommand{\reals}{\mathbb{R}}
\newcommand{\complex}{\mathbb{C}}
\newcommand{\integers}{\mathbb{Z}}
\newcommand{\circel}{\mathbb{S}}
\newcommand{\function}[2]{:#1 \rightarrow #2}
\newcommand{\of}[1]{\left( #1 \right)}
\newcommand{\setdef}[2]{\bigl\{ \hspace{0.5mm} #1 : \hspace{0.5mm} #2 \bigr\}}
\newcommand{\refeq}[1]{(\ref{eq:#1})}
\newcounter{ax}
\newcommand{\axiom}[1]{\refstepcounter{ax}\label{ax:#1}\refax{#1}}
\newcommand{\refax}[1]{(${\mathcal A}_{\ref{ax:#1}}$)}
\newcommand{\refaxxx}[1]{(${\mathcal{A'}}_{\ref{ax:#1}}$)}
\newcommand{\sym}[1]{\mathit{Sym}(#1)}
\newcommand{\e}{\varepsilon}
\newtheorem{theorem}{Theorem}[section]
\newtheorem{lemma}[theorem]{Lemma}
\newtheorem{example}[theorem]{Example}
\newtheorem{definition}[theorem]{Definition}
\newenvironment{proofof}[1]{\par\smallbreak\noindent{\bf Proof of~#1.~}}
{\unskip\nobreak\hfill \qed \par\medbreak}
\newcounter{oq}
\newcommand{\que}{\refstepcounter{oq}\par{\bf \theoq.}~}
\title{Fermat's spiral and the line between Yin and Yang}

\author{Taras Banakh\,$^*$}\thanks{$^*$\,%
Department of Mechanics and Mathematics, Lviv National University,
Universytetska 1, Lviv 79000, Ukraine.}

\author{Oleg Verbitsky\,$^\dag$}\thanks{$^\dag$\,%
Institute for Applied Problems of Mechanics and Mathematics,
Naukova St.\ 3-\usebox{\ukrB},
Lviv 79060, Ukraine.}

\author{Yaroslav Vorobets\,$^\ddag$}\thanks{$^\ddag$\,%
Department of Mathematics, Texas A\&M University
College Station, TX 77843-3368, USA.}

\date{}

\begin{document} 

\begin{abstract}
Let $D$ denote a disk of unit area. We call a set $A\subset D$ \emph{perfect}
if it has measure $1/2$ and, with respect to any axial symmetry of $D$,
the maximal symmetric subset of $A$ has measure $1/4$.
We call a curve $\beta$ in $D$ an \emph{yin-yang line} if
\begin{itemize}
\item
$\beta$ splits $D$ into two congruent perfect sets,
\item
$\beta$ crosses each concentric circle of $D$ twice,
\item
$\beta$ crosses each radius of $D$ once.
\end{itemize}
We prove that Fermat's spiral is a unique yin-yang line in the class of smooth curves
algebraic in polar coordinates.
\end{abstract}

\maketitle

\markleft{\sc TARAS BANAKH, OLEG VERBITSKY, AND YAROSLAV VOROBETS}

\section{Introduction}
The yin-yang concept comes from ancient Chinese philosophy.
\emph{Yin} and \emph{Yang} refer to the two fundamental forces ruling
everything in nature and human's life. The two categories are 
opposite, complementary, and intertwined. They are
depicted, respectively, as the
dark and the light area of the well-known yin-yang symbol
(also \emph{Tai Chi} or \emph{Taijitu}, see Fig.\ \ref{fig:gallery}).
The borderline between these areas represents in Eastern
thought the equilibrium between Yin and Yang.

From the mathematical point of view, the yin-yang symbol is a bipartition
of the disk by a certain curve $\beta$, where by \emph{curve} we mean the image
of a real segment under an injective continuous map. We aim at identifying this curve,
up to deriving an explicit mathematical expression for it.
Such a project should apparently begin with choosing a set of axioms for
basic properties of the yin-yang symbol in terms of $\beta$.
We depart from the following three.

\begin{enumerate}
\item[\axiom{cong}]{\it
$\beta$ splits $D$ into two congruent parts.}
\item[\axiom{conc}]{\it
$\beta$ crosses each concentric circle of $D$ twice.}
\item[\axiom{rad}]{\it
$\beta$ crosses each radius of $D$ once (besides the center of $D$, which
must be visited by $\beta$ due to \refax{conc}).}
\end{enumerate}

While the first two properties are indisputable, our choice of the third axiom
requires some motivation as this condition is often not met in
the actual yin-yang design. Say, in the most familiar versions
some radii do not cross $\beta$ at all (Fig.\ \ref{fig:gallery}-a,b).
On the other hand, there occur yin-yang patterns where the number of crossings can
vary from 1 to 2 (Fig.\ \ref{fig:gallery}-d). Figure \ref{fig:gallery}-c
shows an instance with \refax{rad} obeyed, that arose from attempts
to explain the origins of the yin-yang symbolism \cite{Shac,Tsai}.%
\footnote{The border between Yin and Yang consists of two semicircles in Fig.\ \ref{fig:gallery}-b
and is Archimedes's spiral in Fig.\ \ref{fig:gallery}-c;
the discussion of Fig.\ \ref{fig:gallery}-a,d is postponed to Appendix \ref{app:flags}.}
All the variations can be treated uniformly if we take a dynamic look at the subject.
It will be deduced from our axiomatics that $\beta$ is a spiral.
We therefore can consider a continuous family of yin-yang symbols determined by the spiral
as shown in Fig.\ \ref{fig:evolution}. Thus, \refax{rad} specifies a
single representative that will allow us to expose the whole family.

\begin{figure}
\centerline{\includegraphics{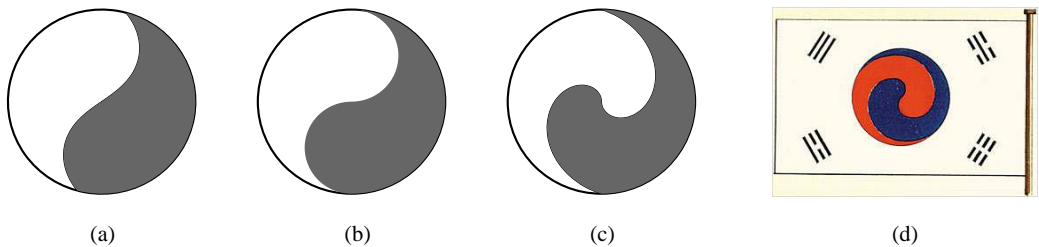}}
\caption{A gallery of yin-yang symbols: (a) a clasical version, cf.\ \cite{Britan};
(b) a modern version; (c) the version presented in \cite{Shac,Tsai};
(c) a Korean flag from 19th century \cite{flag}.}
\label{fig:gallery}
\end{figure}

There is a huge variety of curves satisfying \refax{cong}--\refax{rad}
and hence an essential further specification is needed. We suggest another
axiom, which is both mathematically beautiful and philosophically meaningful.
It comes from continuous Ramsey theory.

\begin{figure}
\centerline{\includegraphics{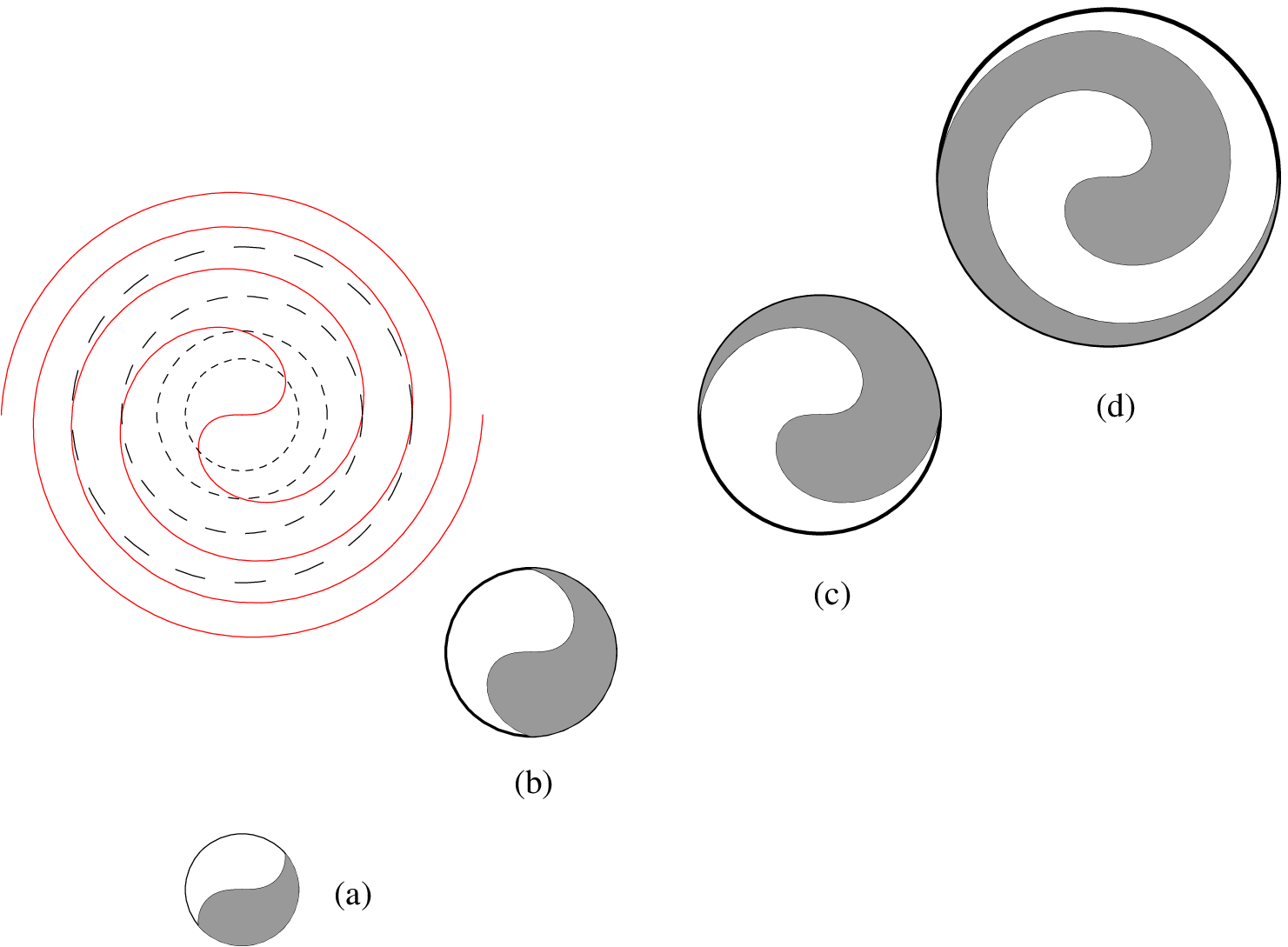}}
\caption{Evolution of the yin-yang symbol (cf.\ Fig.~\protect\ref{fig:gallery}). 
Phase (c) is canonized by axioms \protect\refax{cong}--\protect\refax{alg}.}
\label{fig:evolution}
\end{figure}

Let $\Omega$ be a bounded figure of revolution in a Euclidean space.
Suppose that $\mu(\Omega)=1$, where $\mu$ denotes the Lebesgue measure.
Consider the set $\sym\Omega$ of symmetries of $\Omega$, i.e., of
all space isometries taking $\Omega$ onto itself. Call a set $X\subseteq\Omega$
\emph{symmetric} if $s(X)=X$ for some non-identity $s\in\sym\Omega$.
The density version of a simple Ramsey-type result says that every 
measurable set $A\subseteq\Omega$ contains a symmetric subset of measure
at least $\mu(A)^2$ (see our previous work \cite{Izv}, a survey article \cite{EJC}, or
an overview in Section \ref{s:ramsey}). For the case that $\Omega$ is
the circle in $\reals^2$ or the sphere or the ball in $\reals^n$, $n\ge3$,
we proved an even stronger fact \cite{Izv,EJC}: If $0<\mu(A)<1$, then
$A\subset\Omega$ always contains a symmetric subset whose measure is strictly
more than $\mu(A)^2$. 

On the other hand, we discovered a magic difference
of the plane disk $\Omega=D$ from the aforementioned figures: it contains a set
$A\subset D$ of measure $1/2$ all whose symmetric subsets have measure
at most $1/4$.
Let us call such an $A$ \emph{perfect}. Any perfect set $A$ has a remarkable
property: For \emph{every} axial symmetry $s$ of $D$, the maximum subset of
$A$ symmetric with respect to $s$ has measure $1/4$.
In words admitting far-going esoteric interpretations, perfect sets demonstrate
a sharp equilibrium between their ``symmetric'' and ``asymmetric'' parts,
whatever particular symmetry $s$ is considered.
We put this phenomenon in our axiom system.

\begin{enumerate}
\item[\axiom{perf}]{\it
$\beta$ splits $D$ into perfect sets (from now on we suppose that $D$
has unit area).}
\end{enumerate}

There is still a lot of various examples of
curves satisfying \refax{cong}--\refax{perf}. Nevertheless, we are able 
to show that these four axioms determine a unique curve $\beta$
after imposing two other natural conditions.

\begin{enumerate}
\item[\axiom{smooth}]{\it
$\beta$ is smooth, i.e., has an infinitely differentiable
parameterization $\beta\function{[0,1]}D$ with nonvanishing derivative.}
\end{enumerate}

\noindent
The other condition expresses a belief that the border between Yin and Yang
should be cognizable, i.e., it should not be transcendent neither in the 
philosophical nor in the mathematical sense.

\begin{enumerate}
\item[\axiom{alg}]{\it
$\beta$ is algebraic in polar coordinates.}
\end{enumerate}

By \emph{polar coordinates} we mean the mapping $\Pi$
from the two-dimensional space of parameters $(\phi,r)$ onto two-dimensional
space of parameters $(x,y)$ defined by the familiar relations
$x=r\cos\phi$ and $y=r\sin\phi$. The $(x,y)$-space is considered the standard
Cartesian parameterization of the plane. Note that, like $x$ and $y$,
both $\phi$ and $r$ are allowed to take on any real value. A curve
in the $(\phi,r)$-plane is \emph{algebraic} if all its points satisfy an equation
$P(x,y)=0$ for some bivariate nonzero polynomial $P$. A curve in the $(x,y)$-plane
is \emph{algebraic in polar coordinates} if it is the image of
an algebraic curve under the mapping~$\Pi$.

A classical instance of a curve both smooth and algebraic in polar coordinates
is \emph{Fermat's spiral} (exactly this spiral is drown in 
Fig.\ \ref{fig:evolution}). Fermat's spiral is defined by equation $a^2r^2=\phi$.
Its part specified by the
restriction $0\le\phi\le\pi$ (or, equivalently, $-\sqrt\pi/a\le r \le\sqrt\pi/a$)
is inscribed in the disk of area $(\pi/a)^2$.

\begin{theorem}\label{thm:yy}
Fermat's spiral $\pi^2r^2=\phi$ is, up to congruence, a unique curve satisfying
the axiom system \refax{cong}--\refax{alg}.
\end{theorem}

Thus, if we are willing to accept axioms \refax{cong}--\refax{alg},
the yin-yang symbol must look as in Fig.\ \ref{fig:evolution}.
Note that the factor of $\pi^2$ in the curve equation in Theorem \ref{thm:yy}
comes from the condition that $D$ has area 1. 
As all Fermat's spirals are homothetic, we can equally well draw the yin-yang symbol using, 
say, the spiral $\phi=r^2$. Varying the range of $r$, we obtain modifications as twisted 
as desired (see a MetaPost code in Appendix~\ref{app:code}).

In the next section we overview relevant facts from continuous Ramsey theory.
Theorem \ref{thm:yy} is proved in Section \ref{s:yy}. In Section \ref{s:minimality}
we discuss modifications of the axiom system \refax{cong}--\refax{alg}.
In particular, we show that, after removal of \refax{smooth} or \refax{alg},
the system is not any more categorical, that is,
does not specify the line between Yin and Yang uniquely.

\section{Density results on symmetric subsets}\label{s:ramsey}

We have mentioned several Ramsey-type results on symmetric subsets as a
source of motivation for \refax{perf}, the cornerstone of our yin-yang
axiomatics. Now we supply some details. Recall that we consider any
geometric figure $\Omega\subset\reals^n$ along with its symmetry group
$\sym\Omega$ consisting of the isometries of $\reals^n$ that take $\Omega$
onto itself. Given a measurable set $A\subseteq\Omega$, we address subsets
of $A$ symmetric with respect to a symmetry in $\sym\Omega$ and want to know
how large they can be. To prove the existence of a large symmetric subset,
it sometimes suffices to restrict one's attention to symmetries in a proper
subset of $\sym\Omega$. In the case of the circle we will focus on its
axial symmetries. After identification of the circle with the additive
group $\circel=\reals/\integers$, these symmetries are expressed in the form
$s_g(x)=g-x$.

Any figure of revolution $\Omega$ can be represented as the product 
$\Omega=\circel\times\Theta$ for an appropriate $\Theta$.
The Lebesgue measure on $\Omega$ is correspondingly decomposed into
$\mu=\lambda\times\nu$. It is supposed that 
$\lambda(\circel)=1$ and $\nu(\Theta)=1$.
Note that the $\Omega$ inherits the reflectional symmetries of $\circel$, namely
$s_g(x,y)=(g-x,y)$, where $x\in\circel$ and $y\in\Theta$. 

\begin{theorem}[\cite{Izv,EJC}]\label{thm:L}
Let $\circel\times\Theta$ be a figure of revolution. Then
every measurable set $A\subseteq\circel\times\Theta$ contains a symmetric subset
of measure at least $\mu(A)^2$.
\end{theorem}

\begin{proof}
For a fixed $y\in\Theta$, let $A_{y}=\setdef{x\in\circel}{(x,y)\in A}$ be the 
corresponding section of $A$. By $\chi_U$ we denote the characteristic function 
of a set $U$. Since $s_g$ is involutive, the maximum subset of $A$
symmetric with respect to $s_g$ is equal to $A\cap s_gA$.
Its measure is representable as
$$
\mu(A\cap s_gA)=\int_\Omega\chi_{A\cap s_gA}(x,y)\,d\mu(x,y)=
\int_{\Theta}\int_\circel\chi_{A_{y}}(x)\chi_{A_{y}}(g-x)\,d\lambda(x)d\nu(y).
$$ 
Averaging it on $g$ and changing the order of integration, we have
$$
\int_{\circel}\mu(A\cap s_gA)\,d\lambda(g)=
\int_{\Theta}\int_{\circel}\chi_{A_{y}}(x)\int_{\circel}\chi_{A_{y}}(g-x)
\,d\lambda(g)d\lambda(x)d\nu(y)
=\int_{\Theta}\lambda(A_{y})^2\,d\nu(y).
$$
By the Cauchy-Schwartz inequality,
$$
\int_{\circel}\mu(A\cap s_gA)\,d\lambda(g)\ge
\of{\int_{\Theta}\lambda(A_{y})\,d\nu(y)}^2=\mu(A)^2.
$$
Thus, there must exist a $g\in\circel$ such that $\mu(A\cap s_gA)\ge\mu(A)^2$.
\end{proof}

In some cases Theorem \ref{thm:L} admits a non-obvious improvement.

\begin{theorem}[\cite{Izv,EJC}]\label{thm:SL}
Let $\Omega$ be the circle in $\reals^2$ or the sphere or the ball in $\reals^k$, $k\ge3$.
Suppose that $A\subset\Omega$ is measurable. If $0<\mu(A)<1$, then $A$ contains a symmetric 
subset of measure strictly more than $\mu(A)^2$. 
\end{theorem}

\begin{proof}
We here prove the theorem only for the circle $\circel$. 
The subsequent argument is hardly applicable to the spheres and balls;
the reader is referred to the treatment of these cases in \cite{Izv,EJC}.

Denote $f(g)=\mu(A\cap s_gA)$. We will prove that the sharp inequality 
$f(g)>\mu(A)^2$ must occur at least
for some $g$. Assume to the contrary that $f(g)\le\mu(A)^2$ for all $g$. Applying
the averaging argument as in the proof of Theorem \ref{thm:L}, we obtain
$\int_\circel f(g)\,dg=\mu(A)^2$ and, therefore,
\begin{equation}\label{eq:fconst}
f(g)=\mu(A)^2\mathrm{\ for\ almost\ all\ }g.
\end{equation}

Let $\chi_A(x)=\sum_{n\in\integers}c_ne^{2\pi inx}$
be the Fourier expansion of $\chi_A$ in $L_2(\circel)$.
Note that $f(g)=\int_\circel\chi_A(x)\chi_A(g-x)\,dx$, that is,
$f$ is the convolution $\chi_A\star\chi_A$.
This allows us to determine the Fourier expansion for $f$
from the Fourier expansion for $\chi_A$, namely
$$
f(g)=\sum_{n\in\integers}c_n^2e^{2\pi ing}
$$
in $L_2(\circel)$.
On the other hand, we know from \refeq{fconst} that $f$ is almost everywhere equal to a constant
(i.e., $f$ is a constant function in $L_2(\circel)$).
By uniqueness of the Fourier expansion we conclude that $c_n^2=0$ whenever $n\ne0$.
It follows that $\chi_A(x)=c_0$ almost everywhere. This is possible only
if $c_0=0$ or $c_0=1$, contradictory to the assumption that $0<\mu(A)<1$.
\end{proof}

Somewhat surprisingly, Theorem \ref{thm:SL} fails to be true for the disk.

\begin{definition}\rm
Let $D$ denote the disk of unit area. Call a set $A\subset D$ \emph{perfect}
if $\mu(A)=1/2$ and every symmetric subset of $A$ has measure at most $1/4$.
\end{definition}

Analysis of the proof of Theorem \ref{thm:SL} reveals
a remarkable property of perfect sets.

\begin{lemma}\label{lem:perfect}
For every axial symmetry $s_g$ of $D$, the maximum subset of a perfect set $A$
symmetric with respect to $s_g$ has measure exactly $1/4$, i.e.,
$$
\mu(A\cap s_gA)=1/4\emph{\ for\ all\ }g.
$$
\end{lemma}

\begin{proof}
Let $f(g)=\mu(A\cap s_gA)$. The lemma follows from \refeq{fconst} and the continuity of $f$.
We only have to prove the latter.

Given an $\varepsilon>0$, choose a compact set $K\subset A$ and an open subset $U\subset D$ so
that $K\subset A\subset U$ and $\mu(U)-\mu(K)<\e/2$. 
Let $r_h$ denote the rotation of $D$ by angle $2\pi h$.
For any $\epsilon>0$ there is a $\delta>0$ such that, if $|h|<\delta$,
then the distance between $r_hx$ and $x$ is less than $\epsilon$ for every $x\in D$.
Using this and the compactness of $K$, we can take $\delta>0$ so small that
$r_hK\subset U$ whenever $|h|<\delta$. Let $g'=g+h$. Assuming that $|h|<\delta$, we have
\begin{eqnarray*}
|f(g')-f(g)|&\le&\mu(s_{g'}A\setminus s_gA)+\mu(s_gA\setminus s_{g'}A)=
\mu(A\setminus r_hA)+\mu(A\setminus r_{-h}A)\\
&\le&\mu(U\setminus r_hK)+\mu(U\setminus r_{-h}K)=
2(\mu(U)-\mu(K))<\e,
\end{eqnarray*}
which witnesses the continuity of $f$ at~$g$.
\end{proof}

\begin{theorem}[\cite{Izv,EJC}]\label{thm:perfect}
Perfect sets exist.
\end{theorem}

We precede the proof of Theorem \ref{thm:perfect} with a technical suggestion,
that will be useful also in the next section.
Instead of the disk $D$, it will be practical
to deal with the cylinder $C=\circel\times(0,1]$ supplied with the product measure
$\mu=\lambda\times\lambda$. For this purpose we prick out the center from $D$ and establish 
a one-to-one mapping $F\function DC$ preserving measure and symmetry.
We describe a point in $C$ by a pair of coordinates $(u,v)$ with $0<u\le1$ and
$0<v\le1$, whereas for $D$ we use polar coordinates
$(r,\phi)$ with $0<r\le\pi^{-1/2}$ and $0<\phi\le2\pi$. In these coordinate systems,
we define 
$$
F(r,\phi)=(\frac\phi{2\pi},\pi r^2).
$$
Since $F$ is a diffeomorphism, a set $X\subseteq D$ is measurable iff so is $F(X)$.

To see the geometric sense of $F$, notice that this mapping takes a radius of the disk
onto a longitudinal section of the cylinder. Furthermore, a concentric circle of $D$
is taken onto a cross section of $C$ so that the area  below the section
is equal to the area within the circle. It follows that, if a set $X\subseteq D$
is measurable, then $X$ and $F(X)$ have the same measure.

The correspondence $F$ preserves symmetry in the following sense:
For every $s\in\sym D$ there is an $s'\in\sym C$ such that, for any $X\subseteq D$,
we have $s(X)=X$ iff $s'(F(X))=F(X)$. Specifically, the rotation $s'(u,v)=(u+h,v)$
(resp., the reflection $s'(u,v)=(g-u,v)$)
of the cylinder corresponds to the rotation by angle $2\pi h$
(resp., to the reflection in the diameter $\phi=\pi g$) of the disk.

Finally, it will be practical to identify $\circel$ with $(0,1]$ and replace $C$ with
the square $(0,1]\times(0,1]$ regarded as the development of the cylinder on a plane.

\begin{proofof}{Theorem \ref{thm:perfect}}
The argument presented below is different from that used in \cite{Izv,EJC}.

Using the transformation $F$, we prefer to deal with the cylinder $C$.
We will construct a set $A\subset C$ so that $F^{-1}(A)$ is perfect.
This goal will be achieved by ensuring the following properties of~$A$:
\begin{itemize}
\item
$\mu(A)=1/2$.
\item
$A$ contains no subset symmetric with respect to any rotation.
\item
Every subset of $A$ symmetric with respect to a reflection has measure $1/4$.
\end{itemize}

Given $v\in(0,1]$, we will denote $A_v=\setdef{u\in\circel}{(u,v)\in A}$.
In the course of construction of $A$ we will obey the condition that
\begin{equation}\label{eq:shiftsect}
\circel\setminus A_v=A_v+\frac12
\end{equation}
for every $v$ (i.e., each $A_v$ and its centrally symmetric image make up a partition
of the circle). 

The desired set
$A$ will be constructed by parts, first below the cross section $v=1/2$ and then
above it. There is much freedom for the first part $A'=A\,\cap\,\circel\times(0,\frac12]$.
We choose it to be an arbitrary
set with each slice $A_v$ being a continuous semi-interval of length $1/2$, i.e.,
a semicircle in $\circel$.
The other part $A''=A\,\cap\,\circel\times(\frac12,1]$ is obtained by lifting $A'$ and
rotating it by angle $\pi/2$, that is, $A''=A'+(\frac14,\frac12)$. 
Note that \refeq{shiftsect} is therewith satisfied.
The simplest example of
a perfect set obtainable this way is shown in Fig.~\ref{fig:simplest}.

\begin{figure}
\centerline{\includegraphics{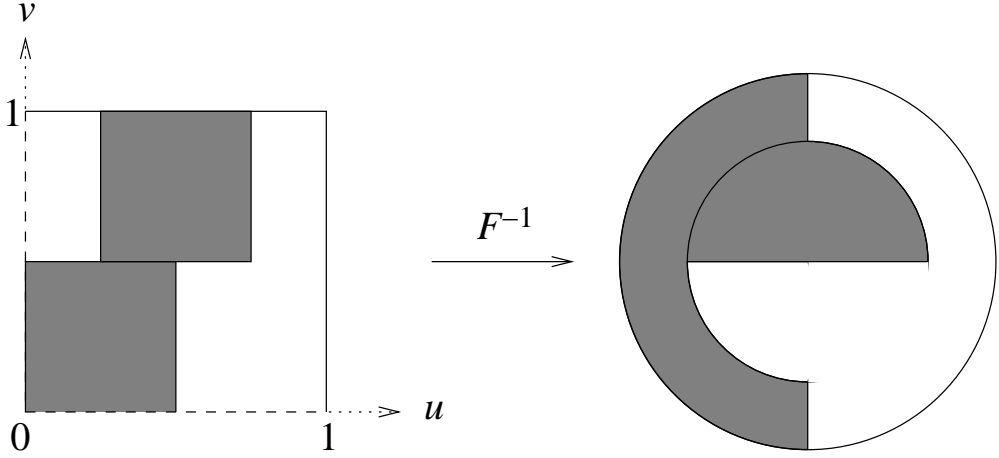}}
\caption{Construction of a perfect set. 
The square shows the cylinder development where we take
$A_v=(0,1/2]$ for $0<v\le1/2$.}
\label{fig:simplest}
\end{figure}

It remains to verify that the $A$ has the proclaimed properties. Since each $A_v$ is a
semicircle, $A$ does not contain any nonempty subset symmetric with respect
to a rotation. Consider now a reflection $s_g(u,v)=(g-u,v)$ and compute the
measure of the maximum subset of $A$ symmetric with respect to $s_g$:
\begin{multline*}
\mu(A\cap s_gA)=\int_0^1\lambda(A_v\cap(g-A_v))\,dv\\
=\int_0^{1/2}\lambda(A_v\cap(g-A_v))+\lambda((\frac14+A_v)\cap(g-\frac14-A_v))\,dv.
\end{multline*}
Note that $\lambda((\frac14+A_v)\cap(g-\frac14-A_v))=\lambda((\frac12+A_v)\cap(g-A_v))$.
Taking into account \refeq{shiftsect}, we conclude that
$$
\mu(A\cap s_gA)=\int_0^{1/2}\lambda(A_v\cap(g-A_v))+
\lambda((\circel\setminus A_v)\cap(g-A_v))\,dv
=\int_0^{1/2}\lambda(g-A_v)\,dv=\frac14.
$$
The proof is complete.
\end{proofof}

\section{\mbox{Establishing the borderline between Yin and Yang}}\label{s:yy}

Here we prove Theorem~\ref{thm:yy}.
The proof essentially consists of Lemmas \ref{lem:alpha} and \ref{main1}
and is given in Section \ref{ss:proofthm}. The proof of Lemma \ref{main1}
is postponed to Section \ref{ss:prooflem} with the necessary preliminaries
outlined in Section~\ref{ss:background}.

\subsection{Proof of Theorem \protect\ref{thm:yy}}\label{ss:proofthm}

Assuming that $\beta$ satisfies axioms \refax{cong}--\refax{rad}, we take
into consideration another curve $\alpha$ associated with $\beta$ in a
natural way.
It easily follows from \refax{cong} and \refax{conc} that $\beta$ is symmetric
with respect to the center $O$ of the disk $D$ and that $O$ splits 
$\beta$ into two congruent branches, say, $\beta_1$ and $\beta_2$. 
We will use the transformation $F\function{D\setminus\{O\}}{(0,1]^2}$ described before the proof
of Theorem \ref{thm:perfect}, where the square $(0,1]^2$ is thought of as the  development
of the cylinder $C=\circel\times(0,1]$. Let $\alpha$ be the image of $\beta_1$ under $F$. 
By \refax{rad}, this curve can be considered the graph of a function
$v=\alpha(u)$, where $u$ ranges in a semi-interval of length $1/2$. 
Shifting, if necessary, the coordinate system and appropriately choosing the positive
direction of the $u$-axis, we suppose that $\alpha$ is defined 
for $0<u\le1/2$ and hence $\lim_{u\to0}\alpha(u)=0$. 
Since $\beta_1$ goes from the center to a peripheral point,
we have $\alpha(1/2)=1$. Since $\beta_1$ crosses each concentric circle exactly once,
$\alpha$ is a monotonically increasing function. Note that the image of $\beta_2$ under $F$
is obtained from $\alpha$ by translation in $1/2$, that is, it is the graph of the function
$v=\alpha(u-1/2)$ on $1/2<u\le1$ (see the left square in Fig.~\ref{fig:alpha}).

\begin{lemma}\label{lem:alpha}
Suppose that $\beta$ satisfies \refax{cong}--\refax{rad} and let $\alpha$ be
the associated curve. Then $\beta$ satisfies \refax{perf} if and only if
\begin{equation}\label{eq:alal}
\alpha(u+\frac14)=\alpha(u)+\frac12\mathrm{\ for\ all\ }0<u\le\frac14.
\end{equation}
\end{lemma}

\begin{proof}
The curve $\alpha$ and its shift in angle $\pi$ split the cylinder $C$ into two 
congruent parts, say, $A$ and $B$. Those are the images under $F$ of the two 
parts into which $D$ is split by $\beta$.

If \refeq{alal} holds true, then $A$ can be considered an instance of the
general construction described in the proof of Lemma \ref{lem:perfect}.
It follows that $F^{-1}(A)$ and $F^{-1}(B)$ are perfect sets, that is, \refax{perf}
is fulfilled.

Conversely, assume \refax{perf}. Consider the development of the cylinder
onto the square and suppose that
$A$ is the area in $(0,1]^2$ between $\alpha$ and its shift in $1/2$.
Introduce first some notation.
Given $v\in(0,1]$, let $A_v=\setdef{u}{(u,v)\in A}$. Note that $A_v$ is the segment with
endpoints $\alpha^{-1}(v)$ and $\alpha^{-1}(v)+1/2$ and we can sometimes write
$A_v=\alpha^{-1}(v)+H$, where $H=[0,\frac12]$.

Furthermore, denote $\bar\alpha(u)=\alpha(u/2)$, $0<u\le1$, and let $\bar A$ be the area of 
the cylinder $C$ bounded by the graph of $\bar\alpha$ and its image
under rotation by angle $\pi$. In the development of $C$ this area looks as shown
on the right square in Fig.\ \ref{fig:alpha}.
Given $u\in(0,1]$, let $\bar A_u=\setdef{v}{(u,v)\in\bar A}$ and $m(u)=\lambda(\bar A_u)$.
If $0<u\le1/2$, we have
\begin{equation}\label{eq:m}
\begin{array}{rcl}
m(u)&=&\bar\alpha(u)+(1-\bar\alpha(u+\frac12)),\\[2mm]
m(u+\frac12)&=&\bar\alpha(u+\frac12)-\bar\alpha(u).
\end{array}
\end{equation}

\begin{figure}
\centerline{\includegraphics{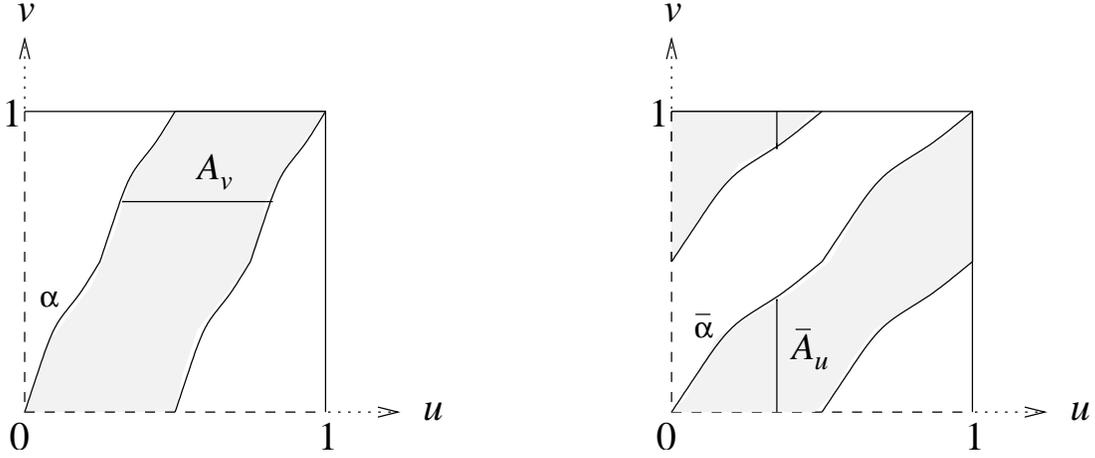}}
\caption{Proof of Lemma \protect\ref{lem:alpha}: areas $A$ and $\bar A$ in the cylinder $C$.}
\label{fig:alpha}
\end{figure}

We are now prepared to prove \refeq{alal}. By \refax{perf} and Lemma \ref{lem:perfect},
we have
$$
\mu(A\cap s_gA)=\frac14
$$
for any reflection $s_g(u,v)=(g-u,v)$, where $\mu=\lambda\times\lambda$
is the measure on $(0,1]^2$.
Observe that
\begin{multline*}
\lambda(A_v\cap s_gA_v)=\lambda((\alpha^{-1}(v)+H)\cap(-\alpha^{-1}(v)+g-H))\\[1mm]
=\lambda((2\alpha^{-1}(v)+H)\cap(g-H))=\lambda((\bar\alpha^{-1}(v)+H)\cap(g-H))
\end{multline*}
Suppose that $1/2<g\le1$ and continue this chain of equalities:
$$
\lambda(A_v\cap s_gA_v)
=\int_{g-1/2}^g\chi_{\bar\alpha^{-1}(v)+H}(u)\,du=
\int_{g-1/2}^g\chi_{\bar A}(u,v)\,du=\int_{g-1/2}^g\chi_{\bar A_u}(v)\,du.
$$
Using it, we obtain
\begin{multline*}
\mu(A\cap s_gA)=\int_0^1\lambda(A_v\cap s_gA_v)\,dv=
\int_0^1\int_{g-1/2}^g
\chi_{\bar A_u}(v)
\,dudv  \\
=\int_{g-1/2}^g\int_0^1
\chi_{\bar A_u}(v)
\,dvdu
=\int_{g-1/2}^gm(u)\,du.
\end{multline*}
Since $\mu(A\cap s_gA)$ is a constant function of $g$, 
differentiation in $g$ gives us the identity
$m(g)-m(g-1/2)=0$, that is,
\begin{equation}\label{eq:mm}
m(u)=m(u+\frac12)\mathrm{\ for\ all\ }0<u\le\frac12.
\end{equation}
Substituting from \refeq{m}, we get
$$
2\bar\alpha(u+\frac12)=2\bar\alpha(u)+1.
$$
In particular, for all $0<u\le1/4$ we have
$$
\bar\alpha(2u+\frac12)=\bar\alpha(2u)+\frac12,
$$
which in terms of $\alpha$ is exactly what was claimed.
\end{proof}

The simplest curve satisfying \refeq{alal} is the line $v=2u$, $0\le u\le1/2$.
As easily seen, under the application of $F^{-1}$ this line and its shift
$v=2u-1$, $1/2\le u\le1$, are transformed into Fermat's spiral $\pi^2r^2=\phi$, $0\le\phi\le\pi$.
It follows that Fermat's spiral, which clearly satisfies \refax{cong}--\refax{rad},
also satisfies \refax{perf}. Note that \refax{smooth} and \refax{alg} hold true
as well: Fermat's spiral is explicitly given by an algebraic relation in polar
coordinates and is smooth, for example, due to parameterization
$$
\left\{
\begin{array}{lll}
x&=&\textstyle\frac1\pi\, t\cos t^2,\\[1.5mm]
y&=&\textstyle\frac1\pi\, t\sin t^2.
\end{array}
\right.
$$

To prove Theorem \ref{thm:yy}, we hence have to prove that, if a curve $\beta$
satisfies \refax{cong}--\refax{alg}, then it coincides, up to congruence,
with Fermat's spiral. It suffices to prove that the associated curve $\alpha$
is a straight line segment.

Since $\beta$ is smooth by \refax{smooth} and $\alpha$ is a smooth transformation
of a part of $\beta$, we conclude that $\alpha$ is smooth as well. 
As follows from \refax{alg}
and the definition of transformation $F$, all points of $\alpha$ satisfy relation
$P(u,\sqrt v)=0$ for a nonzero bivariate polynomial $P$. 
Since this relation can be rewritten
in the form $\sqrt v\, Q(u,v)+R(u,v)=0$ for some polynomials $Q$ and $R$, the points
of $\alpha$ satisfy polynomial relation $v\,Q^2(u,v)-R^2(u,v)=0$.

Thus, $\alpha$ is a smooth algebraic curve. By Lemma \ref{lem:alpha}, $\alpha$
satisfies relation \refeq{alal}. This means that $\alpha$ has infinite
intersection with its shift $\alpha+(\frac14,\frac12)$. Theorem \ref{thm:yy}
immediately follows from the following fact.

\begin{lemma}\label{main1} 
If a smooth algebraic curve $\gamma\subset\reals^2$ 
has infinite intersection with its shift in a nonzero vector $(a,b)$,
then $\gamma$ is contained in a line.
\end{lemma}

The proof of Lemma \ref{main1} takes the rest of this section.

\subsection{Background on smooth algebraic curves}\label{ss:background}

In what follows we suppose that a curve $\gamma\subset\reals^2$ is the image
of a real interval under an injective continuous map
(the forthcoming considerations apply to $\alpha$, which is the image
of a semi-interval, after removal from $\alpha$ its endpoint).
Recall that $\gamma$ is \emph{smooth} if it has an infinitely differentiable
parameterization $\gamma\function{(0,1)}{\reals^2}$ with nonvanishing derivative.
A $\gamma$ is \emph{algebraic} if every point $(x,y)\in\gamma$ satisfies
relation $P(x,y)=0$ for a nonzero polynomial $P$ with real coefficients. 
The \emph{degree} of $\gamma$
is the minimum possible degree of such a $P$. If $P$ is irreducible,
i.e., does not admit factorization $P(x,y)=Q(x,y)R(x,y)$ with $Q$ and $R$
both being nonconstant real polynomials, then $\gamma$ is called 
\emph{irreducibly algebraic}.

\begin{theorem}[B\'ezout, see, e.g., \cite{Kir}]
Suppose that $P_1$ and $P_2$ are distinct bivariate polynomials of degree $m$ and $n$,
respectively. Let $\rho_i=\setdef{(x,y)\in\reals^2}{P_i(x,y)=0}$, $i=1,2$. 
If both $P_1$ and $P_2$
are irreducible, then $\rho_1$ and $\rho_2$ have at most $mn$ points in common.
\end{theorem}

A function $f\function\reals\reals$ (resp., $f\function\complex\complex$) is called
\emph{analytic at point} $x_0$ if in some neighborhood of
$x_0$ the function develops in a power series
$$
f(x) = \sum_{n=0}^\infty a_n ( x-x_0 )^n
$$
with real (resp., complex) coefficients $a_n$.
Furthermore, we call $f$ \emph{analytic on an open set} $U$ if
it is analytic at every $x_0\in U$.
It can be showed that, if $f$ is analytic at a point $x_0$, then it is analytic
on a neighborhood of $x_0$.
The following noticeable fact applies to both real and complex analytic functions.

\begin{theorem}[The Uniqueness Theorem, see, e.g., \cite{NNi}]
Let functions $f$ and $g$ be analytic on a connected set $U$.
Suppose that the set $\setdef{x\in U}{f(x)=g(x)}$ has an accumulation point inside $U$.
Then $f(x)=g(x)$ everywhere on~$U$.
\end{theorem}

We will need also the following result.

\begin{theorem}[\cite{FL}]\label{FL}
Let $P(z,w)$ be a bivariate complex polynomial that essentially depends on $w$.
Suppose that $P(0,0)=0$. Then there are a positive integer $m$ and
complex functions $f_1(z),\dots,f_k(z)$ analytic on a neighborhood of zero $U\subset\complex$ 
such that, for each nonnegative real $t\in U$, every root of the univariate 
polynomial $P(t,\cdot)$
in $U$ is equal to $f_i(\kern-2pt\sqrt[m]{t})$ for some $i\le k$.
\end{theorem}

Finally, we introduce a useful concept stronger than that of a smooth curve.
We call a curve $\gamma\subset\reals^2$ {\em analytic} if for each point $(x_0,y_0)\in\gamma$
we can choose a Cartesian coordinate system so that, within a neighborhood of $(x_0,y_0)$,
the curve $\gamma$ becomes the graph of a real function $y=f(x)$ analytic at~$x_0$.

\subsection{Proof of Lemma \protect\ref{main1}}\label{ss:prooflem}

We begin with a simpler particular case. Assume that $\gamma$ is irreducibly algebraic.

Let $P(x,y)$ be an irreducible polynomial such that $\gamma\subseteq\rho$ where 
$\rho=\{(x,y)\in\reals^2\,:\,P(x,y)=0\}$. 
Note that $\rho+(a,b)=\setdef{(x,y)\in\reals^2}{P(x-a,y-b)=0}$ and that the polynomial $P(x-a,y-b)$
is irreducible as well. Since the intersection 
$\rho\cap (\rho+(a,b))\supseteq\gamma\cap(\gamma+(a,b))$ is infinite,
the B\'ezout Theorem implies equality $\rho=\rho+(a,b)$. It follows that $\rho$ 
has infinite intersection with line $\ell=\setdef{(x,y)\in\reals^2}{b(x-x_0)-a(y-y_0)=0}$
for an arbitrary $(x_0,y_0)\in\rho$. 
Since any linear polynomial is irreducible, we can apply the B\'ezout Theorem once again
and conclude that $\ell=\rho\supset\gamma$.

To prove Lemma \ref{main1}, it now suffices to show that every smooth algebraic curve 
is irreducibly algebraic. We split this task into two steps by proving first that
every smooth algebraic curve is analytic and then that every analytic algebraic 
curve is irreducibly algebraic. 
The rest of the proof of Lemma \ref{main1} consists of the two respective lemmas.

\begin{lemma}\label{l3} 
Every smooth algebraic curve is analytic.
\end{lemma}

\begin{proof} 
Let $(x(s),y(s))$, $0<s<1$, be a smooth algebraic curve.
We need to check that it is analytic at each point $s_0\in(0,1)$. We lose no generality 
assuming that $x(s_0)=y(s_0)=0$. By a suitable rotation of the coordinate system, 
we can reduce our problem to the case that in some
neighborhood of $(0,0)$ the curve coincides with the graph of a $C^\infty$-function 
$y(x)$ defined on an interval $(-\delta,\delta)$. 
It suffices to verify that the function $y(x)$ is analytic at zero. 

In fact, it suffices to find $\e>0$ and two power series $\sum_{n=0}^\infty a_nx^n$ and
$\sum_{n=0}^\infty b_nx^n$ that converge to the function $y(x)$ on the intervals $[0,\e)$ 
and $(-\e,0]$ respectively. 
Indeed, one can use standard analytic tools to show that
\begin{equation}\label{eq:taylor}
a_n=\frac{y^{(n)}(0)}{n!}=b_n,
\end{equation} 
which means that the two series  
determine the same analytic function coinciding with $y(x)$ on $(-\e,\e)$.

Now we show how to find $\e$ and the series $\sum_{n=0}^\infty a_nx^n$ (the series 
$\sum_{n=0}^\infty b_nx^n$ exists by a symmetric argument).

Since our smooth curve $(x,y(x))$ is algebraic, there is a polynomial $P(x,y)$ 
such that $P(x,y(x))=0$ for all $x\in(-\delta,\delta)$. Note that $P$ must contain
at least one occurrence of the variable $y$.
Let a neighborhood $U\subset\complex$ of zero, functions $f_1,\ldots,f_k$ analytic on $U$, 
and a positive integer $m$ be as granted by Theorem \ref{FL}. Suppose that the $f_i$'s
are pairwise distinct. 
Take a positive real $\e<\delta$ so small that 
$\setdef{(x,y(x))}{|x|<\e}\subset U^2$ and 
\begin{equation}\label{eq:noneq}
f_i(z)\ne f_j(z)\mathrm{\ whenever\ }0<|z|<\sqrt[m]{\e},
\end{equation}
for any $1\le i<j\le k$.
Such a choice of $\e$ exists by the Uniqueness Theorem.
Non-equality \refeq{noneq} ensures disjointness of sets
$$
W_i=\setdef{(x,f_i(\kern-2pt\sqrt[m]{x}))}{0<x<\e}\mathrm{\ and \ } 
W_j=\setdef{(x,f_j(\kern-2pt\sqrt[m]{x}))}{0<x<\e}.
$$

By Theorem \ref{FL} we have inclusion 
$\setdef{(x,y(x))}{0< x<\e}\subset \bigcup_{i=1}^k W_i$.
It follows by continuity of the graph of $y(x)$ that
\begin{equation}\label{eq2}
\setdef{(x,y(x))}{0< x<\e}\subset W_i
\end{equation}
for some $i$. 
Develop the analytic function $f_i$ into the Maclaurin series 
$f_i(x)=\sum_{n=0}^\infty a_nx^n$. 
Replacing, if necessary, the number $\e$ by a smaller one, we can 
assume that this series converges to $f_i(x)$ uniformly on the interval 
$[0,\kern-2pt\sqrt[m]{\e})$. 
By (\ref{eq2}) we hence have $y(x)=\sum_{n=0}^\infty a_n x^{n/m}$, where the series
converges uniformly for all $x\in[0,\e)$. 
The smoothness of $y(x)$ 
at zero guarantees that $a_n=0$ for all numbers $n$, not divisible by $m$.
Therefore, $y(x)=\sum_{n=0}^\infty a_{nm}x^n$ for $x\in[0,\e)$, quod erat demonstrandum.
It remains to note that the coefficients actually must be real (say, by \refeq{taylor}).
\end{proof}

\begin{lemma}\label{l4} 
Every analytic algebraic curve $\gamma\subset\reals^2$ is irreducibly\/\footnote{%
The lemma claims the irreducibility over $\reals$. The proof actually
implies the irreducibility even over $\complex$.}
algebraic.
\end{lemma}

\begin{proof}
Let $P(x,y)$ be a nonzero  polynomial of the smallest possible degree such 
that $\gamma\subseteq\setdef{(x,y)\in\reals^2}{P(x,y)=0}$. We will prove that $P(x,y)$ is irreducible. 

Assume to the contrary that $P(x,y)=P_1(x,y) P_2(x,y)$ with $P_1$ and $P_2$ being
non-constant polynomials. For $i=1,2$ let $\gamma_i=\setdef{(x,y)\in\gamma}{P_i(x,y)=0}$ and 
$\gamma_i'$ be the set of non-isolated points of $\gamma_i$. 
It is clear that the sets $\gamma_i$ and $\gamma_i'$ are closed in $\gamma$. 
We claim that $\gamma_i'$ is also open in $\gamma$. 

Indeed, take an arbitrary point $(x_0,y_0)\in\gamma_i'$. After moving appropriately 
the coordinate system, we can assume that in some neighborhood 
of $(x_0,y_0)$ the curve $\gamma$ coincides with the graph of a real function $f(x)$
analytic on an interval $(x_0-\e,x_0+\e)$. Notice that the function $g_i(x)=P_i(x,f(x))$ 
is analytic on the same interval.
Denote the projection of $\gamma_i$ onto the $x$-axis by $\eta_i$.
Since $(x_0,y_0)\in\gamma_i'$, the set $\eta_i\cap (x_0-\e,x_0+\e)$ is infinite and
contains $x_0$ as an accumulation point. By the Uniqueness Theorem, 
$g_i(x)=0$ for all $x\in(x_0-\e,x_0+\e)$. Consequently, $\setdef{(x,f(x))}{x_0-\e<x<x_0+\e}\subset\gamma_i$
and hence $(x_0,y_0)$ belongs to the interior of $\gamma_i'$ in $\gamma$. 

Thus, both the sets $\gamma_1'$ and $\gamma_2'$ are closed-and-open in $\gamma$. 
Since $\gamma=\gamma_1'\cup\gamma_2'$, one of the sets is nonempty and hence 
coincides with $\gamma$ by connectedness of $\gamma$. We see that 
$\gamma=\gamma_i$ is included in $\setdef{(x,y)\in\reals^2}{P_i(x,y)=0}$ for the respective $i$.
This makes a contradiction because $P_i$ has degree strictly less than the degree of~$P$. 
\end{proof}

\section{Modifications of the axiom system \protect\refax{cong}--\protect\refax{alg}}%
\label{s:minimality}

We now investigate what happens when we deviate a little from our axiomatics.
In particular, we prove that if either \refax{smooth} or \refax{alg} is removed,
the yin-yang borderline is not specified any more uniquely by the relaxed axiom system.
We also discuss modifications of \refax{rad} and obtain
an existence and uniqueness result under the
condition that $\beta$ crosses each radius exactly twice: then any spiral
satisfying the other axioms is congruent to the two-turn Fermat spiral.

\subsection{Non-algebraic, analytic Yin-Yang variations}

The following example de\-monstrates that there is a variety of non-congruent curves
$\beta$ satisfying axioms \refax{cong}--\refax{smooth}. By Theorem \ref{thm:yy},
all such curves excepting Fermat's spiral fail to meet~\refax{alg}.

\begin{example}
Let $0<\lambda<1/4$. The polar equation
\begin{equation}\label{eq:sin}
\pi^2r^2=\phi+\lambda\sin4\phi, \mathrm{\ where\ }0\le\phi\le\pi\
(\mathrm{or\ }-\pi^{-1/2}\le r\le\pi^{-1/2}),
\end{equation}
determines an analytic curve $\beta$ satisfying \refax{cong}--\refax{smooth}.
\end{example}

\begin{proof}
For each $0<\lambda<1/4$, the function $\rho(\phi)=\phi+\lambda\sin4\phi$
is strictly monotone and hence invertible. Since $\rho(\phi)$ is analytic,
its inverse $\phi=\phi(\rho)$ is analytic as well. This suggests the
following analytic parameterization for $\beta$:
\begin{equation}\label{eq:sinparam}
\left\{
\begin{array}{rcl}
x(r)&=&r\cos\phi(\pi^2r^2),\\
y(r)&=&r\sin\phi(\pi^2r^2).
\end{array}
\right.
\end{equation}

Note that $\beta$ is inscribed in the disk of unit area.
$\beta$ satisfies \refax{cong} because \refeq{sin} is centrally symmetric.
\refax{conc} follows from the monotonicity of $\rho(\phi)$. Furthermore,
\refax{rad} follows from the fact that $\beta$ is exactly one turn of a spiral.
To verify \refax{perf}, apply the map $F(r,\phi)=(\frac\phi{2\pi},\pi r^2)$
converting the disk into the cylinder. Then $\beta$ is transformed into
two curves, one being the graph of function
$$
\alpha(u)=2\,u+\frac\lambda\pi\,\sin8\pi u.
$$
Since $\alpha$ satisfies \refeq{alal}, property \refax{perf} immediately follows
from Lemma \ref{lem:alpha}. \refax{smooth} is true because the parameterization
\refeq{sinparam} is analytic and hence smooth. 
Finally, $\beta$ is an analytic curve because \refeq{sinparam} allows us to
locally express $y$ as an analytic function of $x$ or $x$ as an analytic function of~$y$.
\end{proof}

\subsection{Finitely differentiable, algebraic Yin-Yang variations}

Now we exhibit non-congruent curves satisfying \refax{cong}--\refax{perf} and
\refax{alg} and being differentiable any preassigned number of times.

\begin{example}
Let $\lambda$ be a positive real and $k$ a nonnegative integer.
Define a function
$$
f(u)=\left\{
\begin{array}{lr}
f_1(u)=2u+\lambda u^{k+1}(\frac14-u)^{k+1},&\quad 0\le u\le\frac14,\\[2mm]
f_2(u)=\frac12+f_1(u-\frac14),&\quad \frac14\le u\le\frac12.
\end{array}
\right.
$$
For each sufficiently small $\lambda$, the polar equation
\begin{equation}\label{eq:f}
\pi r^2=f\of{\frac\phi{2\pi}}, \mathrm{\ where\ }0\le\phi\le\pi\
(\mathrm{or\ }-\pi^{-1/2}\le r\le\pi^{-1/2}),
\end{equation}
determines a $k$ times continuously differentiable curve $\beta$ satisfying axioms
\refax{cong}--\refax{perf} and~\refax{alg}.
\end{example}

\begin{proof}
Let $\lambda$ be so small that $f_1$ has positive derivative on $[0,\frac14]$.
As easily seen, $f$ is strictly monotone and hence invertible.
Note that $f$ is $k$ times continuously differentiable.
A little care is needed only at $u=1/4$, where we have
$$
f_1^{(i)}(1/4)=f_0^{(i)}(1/4)=f_2^{(i)}(1/4)
$$
for $i\le k$ and $f_0(u)=2u$.
It follows that the inverse $f^{-1}$ is $k$ times continuously differentiable
as well. This leads to a $C^k$-parameterization of \refeq{f}:
$$
\left\{
\begin{array}{lcl}
x(r)&=&r\cos\of{2\pi f^{-1}(\pi r^2) },\\
y(r)&=&r\sin\of{2\pi f^{-1}(\pi r^2) },
\end{array}
\right.
$$
where $-1/\sqrt\pi\le r\le1/\sqrt\pi$.

Axioms \refax{cong}--\refax{rad} are easy to verify.
The transformation $F$ of the disk into the cylinder takes the branch of
\refeq{f} for positive $r$ onto the graph of the function $v=f(u)$.
Since $f$ is defined so that $f(u+\frac14)=f(u)+\frac12$ for all $0\le u\le\frac14$,
axiom \refax{perf} follows by Lemma \ref{lem:alpha}.
Finally, $\beta$ is algebraic because its points satisfy relation 
$$
\of{\pi r^2-f_1(\frac\phi{2\pi})}\of{\pi r^2-f_2(\frac\phi{2\pi})}=0
$$
in polar coordinates.
\end{proof}

\subsection{Varying the number of radial crossings}\label{ss:vary_rad}

For the purpose of the current discussion, we will call a curve $\beta$
\emph{spiral} if it is described by a polar equation
\begin{equation}\label{eq:spiral1}
|r|=f(\phi),\qquad 0\le\phi\le\pi\ell,
\end{equation}
where $f$ is a strictly increasing real function. We will say that
$\beta$ \emph{makes $\ell$ turns}, where $\ell$ does not need be
integer.

Let $D$ be the disk of unit area centered at the coordinate origin.
We will suppose that 
\begin{equation}\label{eq:spiral2}
f(\pi\ell)=\frac1{\sqrt\pi}.
\end{equation}
Thus, $\beta$ is inscribed in $D$ and splits $D$ into two parts. 
Since $\beta$ is symmetric with
respect to the origin, we have \refax{cong}. Condition \refax{conc}
follows from the monotonicity of $f$. Axiom \refax{rad} captures the case that
$\beta$ makes exactly one turn. 
Consider the following modification:
\begin{enumerate}
\item[\refaxxx{rad}]{\it
$\beta$ crosses each radius of $D$ twice (besides the center of $D$).}
\end{enumerate}
As easily seen, any spiral satisfying \refaxxx{rad} makes presisely
2 turns. Note that a curve satisfying \refaxxx{rad} is not necessary
a spiral (see Fig.~\ref{fig:nonspiral}).

\begin{figure}
\centerline{\includegraphics{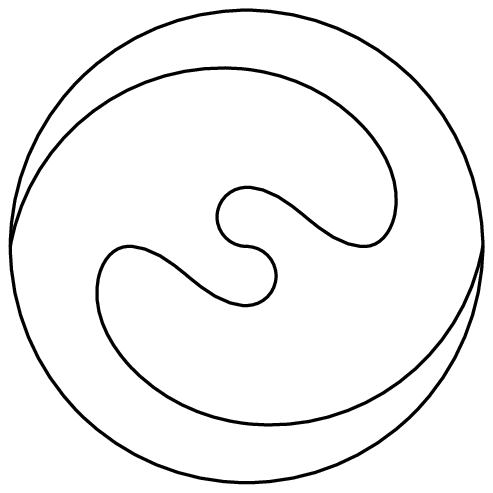}}
\caption{A nonstandard model of axioms \refax{cong}, \refax{conc}, and \refaxxx{rad}.}
\label{fig:nonspiral}
\end{figure}

\begin{theorem}\label{thm:2turns}
A spiral $\beta$ satisfies the axiom system \refax{cong}--\refax{alg} with
\refaxxx{rad} in place of \refax{rad} if and only if it is congruent
to Fermat's spiral
\begin{equation}\label{eq:2turns}
2\pi^2r^2=\phi,\qquad 0\le\phi\le2\pi.
\end{equation}
\end{theorem}

\noindent
The corresponding version of the yin-yang symbol can be seen in Fig.~\ref{fig:evolution}-d.

\begin{proof}
We proceed similarly to the proof of Theorem \ref{thm:yy}.
Let $\alpha$ be the curve associated with $\beta$ in the way
described in the beginning of Section \ref{ss:proofthm}.
We identify this curve with a monotone map
$\alpha\function{(0,1]}{(0,1]}$. Assume that $\beta$ satisfies \refax{perf}.
The proof of Lemma \ref{lem:alpha} applies and gives us the relation \refeq{mm}, which
in terms of $\alpha$ reads
\begin{equation}\label{eq:alalal}
\frac12+\alpha(u)+\alpha(u+\frac12)=
\alpha(u+\frac14)+\alpha(u+\frac34)
\mathrm{\ for\ all\ }0<u\le\frac14.
\end{equation}
Define a function $\sigma\function{(0,\frac34]}{(0,1]}$ by 
$$
\sigma(u)=\alpha(u+\frac14)-\alpha(u).
$$
Relation \refeq{alalal} can be now rewritten as
\begin{equation}\label{eq:sigma}
\sigma(u+\frac12)=\frac12-\sigma(u)\mathrm{\ for\ all\ }0<u\le\frac14.
\end{equation}

Note that $\sigma$ is smooth because so is $\alpha$. Furthermore, recall that
a function $v(u)$ is called \emph{algebraic} if for some bivariate nonzero
polynomial $P$ we have $P(u,v(u))=0$ for all arguments $u$. It is known that
the difference of two algebraic functions is algebraic itself. Thus,
$\sigma$ is algebraic because so is~$\alpha$.

We now need the following analog of Lemma \ref{main1}: A smooth algebraic function $\sigma$
satisfies \refeq{sigma} only if $\sigma$ is constant. By Lemmas \ref{l3} and \ref{l4},
it suffices to prove this claim for an irreducibly algebraic $\sigma$. Let $P$ be
an irreducible polynomial such that $P(u,\sigma(u))=0$ for all $0<u\le3/4$ and denote
$\rho=\setdef{(u,v)}{P(u,v)=0}$. Consider an affine transformation $S(u,v)=(u+\frac12,\frac12-v)$.
Note that $S(\rho)$ consists of exactly those points $(u,v)$ satisfying $P(u-\frac12,\frac12-v)=0$.
Since polynomials $P(u,v)$ and $P(u-\frac12,\frac12-v)$ are irreducible simultaneously
and the intersection $\rho\cap S(\rho)$ is infinite by \refeq{sigma}, we conclude by
B\'ezout's theorem that $\rho=S(\rho)$. It follows that $\rho=S^2(\rho)$. The latter
equality implies that, along with each point $(u_0,v_0)$, the set $\rho$ contains
every point $(u_0+n,v_0)$ for a natural $n$. Therefore, $\rho$ has infinite intersection
with the set of points satisfying a linear equation $v-v_0=0$. Applying B\'ezout's theorem
once again, we see that $\sigma(u)=v_0$ everywhere.

We have just proved that
$$
\alpha(u+\frac14)=\alpha(u)+v_0
\mathrm{\ for\ all\ }0<u\le\frac14,
$$
which provides us with an analog of \refeq{alal}. Applying Lemma \ref{main1}, we
conclude that $\alpha$ is a line and, therefore, $\beta$ is Fermat's spiral.

Verification of conditions \refax{cong}--\refax{alg} for Fermat's spiral \refeq{2turns}
is easy and left to the reader.
\end{proof}

\subsection{Questions and comments}
\mbox{}

\que
Is the axiom system \refax{cong}--\refax{alg} with \refaxxx{rad} in place of \refax{rad}
categorical, that is, does it determine Fermat's spiral uniquely?
By Theorem \ref{thm:2turns}, if these axioms are fulfilled for another curve,
the latter is not a spiral and looks like Fig.~\ref{fig:nonspiral}.

\que
Does the axiom system \refax{cong}--\refax{alg} remain categorical after
exclusion of \refax{conc} from it?

\que
As we mentioned, any spiral (in the sense of \refeq{spiral1}--\refeq{spiral2})
has properties \refax{cong} and \refax{conc}. The following hypothesis seems
plausible: If a spiral satisfies also axioms \refax{perf}--\refax{alg}, then
it is congruent to Fermat's spiral with integer number of turns.
We can confirm this hypothesis for spirals making at most 2 turns.
Consider, for example, the case that a spiral $\beta$ makes $3/2$ turns
and show that such a $\beta$ cannot satisfy all conditions \refax{perf}--\refax{alg}
simultaneously.

If $\beta$ satisfies \refax{perf}, then for the corresponding function
$\alpha\function{(0,\frac34]}{(0,1]}$ we derive from \refeq{mm}, which holds true for any spiral, a relation
\begin{equation}\label{eq:al3}
\alpha(u)+\alpha(u+\frac12)=\alpha(u+\frac14)+\frac12\mathrm{\ for\ all\ }0<u\le\frac14.
\end{equation}
If $\beta$ satisfies also \refax{smooth} and \refax{alg}, $\alpha(u)$ must be analytic.
Using \refeq{al3}, we can analytically extend this function to a larger interval, say,
$(-1/4,1)$. Differentiating repeatedly \refeq{al3}, we obtain a relation
\begin{equation}\label{eq:al4}
\alpha^{(k)}(u)+\alpha^{(k)}(u+\frac12)=\alpha^{(k)}(u+\frac14)
\mathrm{\ for\ all\ }-\frac14<u<\frac12
\end{equation}
for all $k$. Substituting $u=0$ and $u=1/4$ in \refeq{al4}, we easily infer equality
$$
\alpha^{(k)}(3/4)=-\alpha^{(k)}(0).
$$ 
It follows that two analytic functions
$\alpha(u)$ and $-\alpha(u-3/4)$ coincide in a neighborhood of the point $u=3/4$.
By the Uniqueness theorem, we have identity $\alpha(u)=-\alpha(u-3/4)$.
Using it in the same way as the identity \refeq{sigma}
was used in the proof of Theorem \ref{thm:2turns},
we conclude that $\alpha(u)$ must be constant. This contradicts the fact that
$\alpha(u)$ is strictly monotone.

\que
3-part yin-yang symbols are frequently seen in Korea, where they are called
\emph{Sam-Taegeuk}, and in Zen temples in the Himalayas, where they are
called \emph{Gankyil} or \emph{Wheel of Joy} (see Fig.\ \ref{fig:kpartite}-a).
The Gankyil has a 4-part variant \cite{Beer}.
From the mathematical point of view, a $k$-part
yin-yang symbol is a partition of the unit disk $D$ by a $k$-coil spiral into
$k$ congruent parts. Equivalently, we can speak of
a 1-coil spiral $\sigma$ along with its rotations in angles $2\pi i/k$ for all $i<k$.
Recall that the notion of a spiral was formally introduced in the beginning of 
Section \ref{ss:vary_rad};
\emph{1 coil} means that only positive values of $r$ are allowed,
\emph{1 turn} means that $0\le\phi\le\pi$.
Extending the notion of a perfect set to any measurable set $A\subset D$, we call $A$
\emph{perfect} if every symmetric subset of $A$ has measure at most $\mu(A)^2$
(symmetric subsets of measure at least $\mu(A)^2$ always exist by Theorem \ref{thm:L}).
We adopt our axiomatics in the $k$-partite case by requiring that each part of the symbol
is perfect and that $\sigma$ is smooth and algebraic in polar coordinates.
Our main result, Theorem \ref{thm:yy}, generalizes to the $k$-partite case as follows:
Among 1-turn spirals, only Fermat's spiral meets the postulated conditions
(see Fig. \ref{fig:kpartite}-b,c).

\begin{figure}
\centerline{\includegraphics{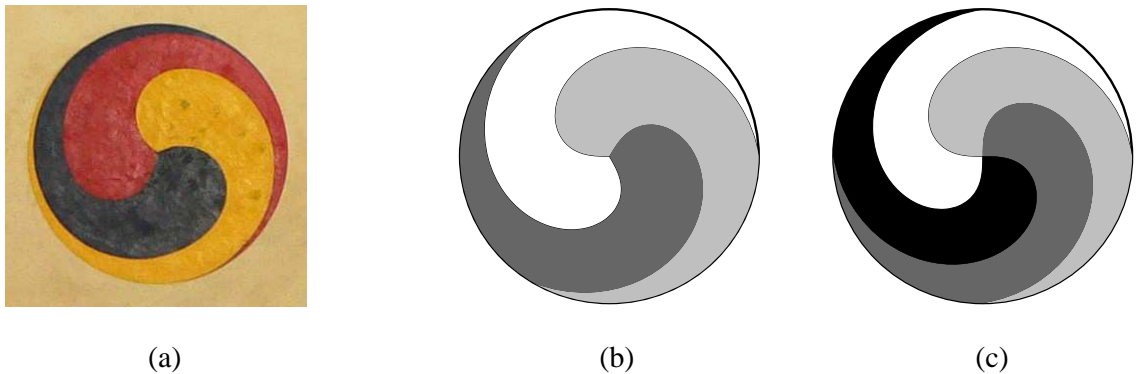}}
\caption{(a) The Sam-Taegeuk symbol \cite{SamTaegeuk}; (b,c)
the versions based on Fermat's spiral.}
\label{fig:kpartite}
\end{figure}

\subsection*{Acknowledgment}
We thank Igor Chyzhykov for referring us to the book~\cite{FL}
and Maria Korolyuk for useful comments.

\clearpage

\appendix

\section{A MetaPost code for drawing yin-yang symbols}\label{app:code}

\begin{verbatim}

beginfig(1)

                     % INPUT PARAMETERS
turn = 1;            % number of turns (does not need to be integer)  
radius = 2cm;        % radius of the disk
rotate = 0;          % optional rotation counter-clockwise (degrees)
w = 1bp;             % thickness of the circle
color dark;          % color of the dark part
dark = .5white;      % (1white=white, 0white=black)
boolean clockwise;   % orientation
clockwise = true;    % set to false if counter-clockwise preferred 

a = ((.5-turn)*180)+rotate;
shif = (radius+1cm); pair shft; shft = (shif,shif);
if (turn > 2) : interpol = (1/(16turn)) else  : interpol = (1/16) fi;
 
path spiral;
spiral := (0,0)
 for r:=0 step interpol until 1:
  ..(r,0) rotated (180r*r*turn)
  endfor
 ..(1,0) rotated (180turn); 

picture symbol;

symbol = image(
  draw spiral scaled radius;
  draw spiral scaled radius rotated 180;
  draw fullcircle scaled (2radius) withpen pencircle scaled w;
  fill buildcycle(reverse spiral scaled radius rotated 180,
                  spiral scaled radius,
                  halfcircle scaled (2radius) rotated (180*(turn-1))) 
                                                   withcolor dark; );

if clockwise:
 draw symbol rotated a shifted shft;
else:
 draw symbol reflectedabout ((0,0),(0,1)) rotated a shifted shft;
fi;

endfig; end

\end{verbatim}

\section{Appearances of Fermat's spiral\\ in the yin-yang symbolism}\label{app:flags}

Playing with Ramsey properties of the plane disk, we have identified the
borderline between Yin and Yang as Fermat's spiral. Though such an approach
must be a novelty in this area, we cannot pretend to come with a new
suggestion in the yin-yang symbolism. It seems that Fermat's spiral
was used to demarcate Yin and Yang already a long time ago.
Two examples are shown in Fig.\ \ref{fig:flags}.
Another example is given by Fig. \ref{fig:gallery}-a:
Being produced by our MetaPost code with parameter {\tt turn=2/9},
it looks quite close to the picture of the yin-yang symbol in \emph{Britannica} 
\cite{Britan}. We find this phenomenon astonishing
and, unfortunaly, are not aware of any historical background for it.

\begin{figure}[h]
\centerline{\includegraphics{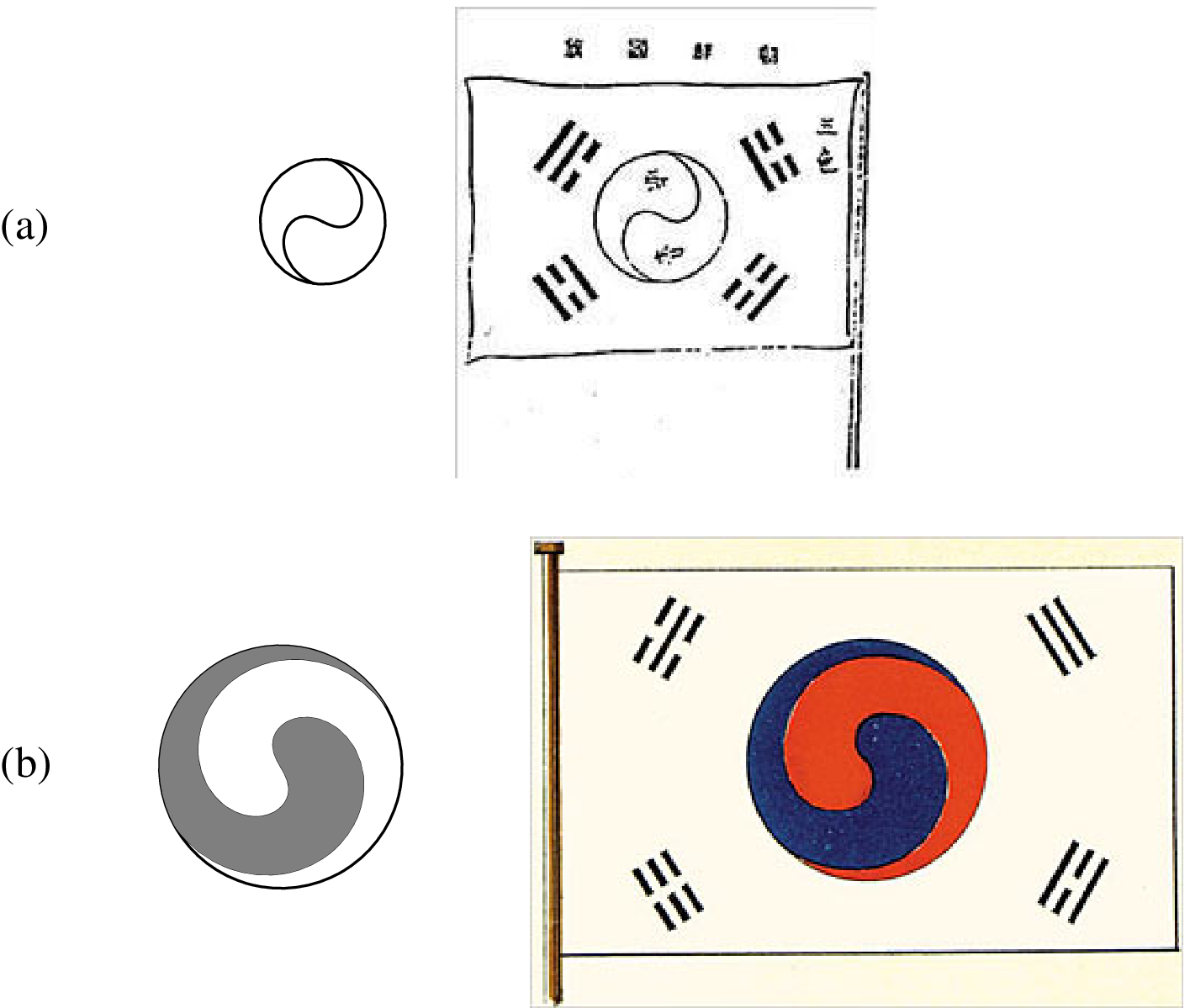}}
\caption{(a) On the right: a flag of Korea's Chosun Dynasty \cite{chosun}.
On the left: the output of the MetaPost code in Appendix \protect\ref{app:code}
with parameters {\tt turn=.6; radius=.75cm; rotate=-8}.\newline\mbox{}\qquad
(b) On the right: the earliest version of the Korean flag \cite{flag}.
On the left: the output of the MetaPost code in Appendix \protect\ref{app:code}
with parameters {\tt turn=1.5; radius=1.465cm; rotate=-60}.}
\label{fig:flags}
\end{figure}


\begin{thebibliography}{1}

\bibitem{Izv}
T.~Banakh, Ya.~Vorobets, and O.~Verbitsky.
\newblock
Ramsey-type problems for spaces with symmetry.
\newblock
{\em Izvestiya RAN, seriya matematicheskaya\/} 64(6):3--40 (2000).
\newblock
In Russian. English translation in:
{\em Russian Academy of Sciences. Izvestiya. Mathematics\/}
64(6):1091--1127 (2000).

\bibitem{EJC}
T.~Banakh, Ya.~Vorobets, and O.~Verbitsky.
\newblock
A Ramsey treatment of symmetry.
\newblock
{\em The Electronic Journal of Combinatorics\/}
7(\#R52):1--25 (2000).

\bibitem{Beer}
R.~Beer.
\newblock
{\em The Handbook of Tibetan Buddhist Symbols.}
Serindia Publications, Inc.\ (2003).

\bibitem{FL} 
B.A.~Fuks, V.I.~Levin. 
\newblock
{\it Functions of a complex variable and their applications.} 
GITTL, Moscow (1951). In Russian.
\newblock 
English translation:
{\it Functions of a complex variable and some of their 
applications II.} 
Pergamon Press, Oxford, and Addison-Wesley, Reading, MA (1961).

\bibitem{Kir} 
F.~Kirwan.
\newblock
{\it Complex algebraic curves.}
\newblock
Cambridge University Press (1992).

\bibitem{NNi}
R.~Narasimhan, Y.~Nievergelt.
\newblock
{\it Complex analysis in one variable.}
\newblock
Springer Verlag (2001). 

\bibitem{Shac}
M.~Shackelford.
\newblock
The T’ai Chi and the cycle of Chinese months.
\newblock
An online publication,
{\sl www.fengshui-magazine.com/Trigrams.htm},
accessed \accessed.

\bibitem{Tsai}
A.~Tsai.
\newblock
Where does the Yin Yang symbol come from?
\newblock
An online publication, 
{\sl www.chinesefor} {\sl tunecalendar.com/yinyang.htm},
accessed \accessed.

\bibitem{flag}
The earliest surviving depiction of the flag of Korea, 
as published in a U.S. Navy book {\it Flags of Maritime Nations\/} 
in July 1882. 
\newblock
Reproduced in: {\it Wikimedia Commons},
{\sl http://en.wikipedia.org/wiki/Image:Taegukgi.jpg},
accessed \accessed.


\bibitem{chosun}
A drawing of a flag of Korea's Chosun Dynasty published in the Japanese newspaper 
\emph{Jijishinpo} on October 10, 1882.
\newblock
Reprinted in: {\it The Chosun Ilbo}, January 26, 2004.
{\sl http://english.chosun.com/w21data/html/news/200401/200401260030.html},
accessed \accessed.

\bibitem{SamTaegeuk}
Sam-Taegeuk.
\newblock
An image from the {\it Wikimedia Commons},
{\sl http://en.wikipedia.org/wiki/File:} {\sl Sam\underline{\ }Taeguk.jpg\#file},
accessed \accessed.
\newblock
Author: Craig Nagy, {\sl http://www.flickr.com/} {\sl photos/nagy/}.
The image is licensed under \emph{Creative Commons Attribution 2.0} License
{\sl http://creativecommons.org/licenses/by/2.0/}

\bibitem{Britan}
Yin-yang. 
\newblock
Article in \emph{The New Encyklop\ae dia Britannica},
15-th edition. Microp\ae dia, Vol.\ 12, page 845 (1991).

\end{thebibliography}
\end{document}